\documentclass[12pt,reqno]{amsart}
\usepackage{graphicx,  amssymb, fullpage, verbatim, amsthm, mathrsfs, paralist, amsfonts, amscd, amsmath, mathtools}
\usepackage[margin=0.8in]{geometry}
\usepackage[colorlinks]{hyperref}
\usepackage[alphabetic]{amsrefs}
\hypersetup{citecolor=blue}
\numberwithin{equation}{section}

\theoremstyle{plain}
\newtheorem{thm}{Theorem}[section]

\newtheorem{prop}[thm]{Proposition}

\theoremstyle{definition}
\newtheorem{defin}{Definition}[section]
\newtheorem{remark}{Remark}[section]

\newcommand{\qede}{\hfill $\diamond$}

\newcommand{\kl}{\mathcal K \mathcal L}
\newcommand{\R}{\mathbb R}

\newcommand{\Z}{\mathbb Z}
\newcommand{\del}{\Delta}
\newcommand{\T}{T_{S,L_{\lambda}}}
\newcommand{\eps}{\varepsilon}
\newcommand{\e}{\mathbf e}
\renewcommand{\int}{\mathrm{int}}
\renewcommand{\S}{\mathcal S}
\renewcommand{\b}{\beta}
\renewcommand{\d}{\mathcal D}

\begin{document}
\title{A remark on the convergence of the Douglas-Rachford iteration in a non-convex setting}

\dedicatory{Dedicated to the memory of Jonathan M. Borwein}

\author{Ohad Giladi}

\address{School of Mathematical and Physical Sciences, University of Newcastle, Callaghan, NSW 2308, Australia}

\email{ohad.giladi@newcastle.edu.au}

\date{\today}

\subjclass[2010]{49N60, 47J25, 39A22}

\begin{abstract}
Using the construction of a Lyapunov function from~\cite{Ben15}, it is shown that the Douglas-Rachford iteration with respect to a sphere and a line in $\R^d$ is robustly $\kl$-stable. This implies a convergence which is stronger than uniform convergence on compact sets.
\end{abstract}

\keywords{Douglas-Rachford iteration, Lyapunov function, robust $\mathcal{KL}$-stability}

\maketitle

\section{Introduction}

Given a set $A\subseteq \R^d$, define the projection operator $P_A:\R^d \rightrightarrows \R^d$,
\begin{align*}
P_A(x) = \Big\{y \in A~\big|~ \|x-y\| = \inf_{z\in A}\|x-z\|\Big\}.
\end{align*}
In general $P_A$ can be multi-valued. Here and in what follows, $\|\cdot\|$ denotes the Euclidean norm on $\R^d$. Define also the reflection operator by $R_A = 2P_A-I$, $I$ being the identity operator on $\R^d$. Given two sets $A,B\subseteq \R^d$ define the Douglas-Rachford operator,
\begin{align}\label{def dr}
T_{A,B} = \frac{I+R_BR_A}{2}.
\end{align}
Given $x \in \R^d$, define the sequence $\{x_n\}_{n=0}^{\infty}\subseteq \R^d$ by the recursive condition
\begin{align}\label{iterate seq}
x_{n+1} = T_{A,B}x_n = T^n_{A,B}x_0, ~~ x_0 = x.
\end{align}
The sequence defined in~\eqref{iterate seq} is known as the \emph{Douglas-Rachford iteration} of $x$. A well known question concerns the asymptotic behaviour of this sequence. This question has application in the case where both $A$ and $B$ are convex, as well as in the case where one of the sets is not convex. See for example~\cites{BCL 02, LM79} for the convex case and~\cites{ERT07, GE08} for the non-convex case.


In the case $A$ is convex, it is known that the projection map $P_A$ is firmly non-expansive, that is, for all $x,y \in \R^d$,
\begin{align*}
\|P_Ax-P_Ay\|^2 + \|(I-P_A)x - (I-P_A)y\|^2 \le \|x-y\|^2.
\end{align*} 
See for example~\cite{GK90}*{Thm.~12.2}. It then follows that if both $A$ and $B$ are convex, $T_{A,B}$ is also firmly non-expansive. See for example~\cite{GK90}*{Thm.~12.1}. From the non-expansiveness of $T_{A,B}$ it follows that $\{x_n\}_{n=0}^{\infty}$ is norm convergent, with norm convergence replaced by weak convergence in the case of an infinite dimensional space.


While the convex case is well understood, much less is known in the non-convex setting, when either $A$ or $B$ is not convex. One of the simplest non-convex cases is the case of a sphere and a line. This case is of particular interest also because the sphere is a model of many reconstruction problems in which only the magnitude of a phase is measured. This case was studied in~\cites{AB13, Ben15, BS11}. Other non-convex cases were considered in~\cite{ABT16, BG16, HL13, Pha16}.


Let $\e_1,\dots,\e_n$ be the standard basis vectors in $\R^d$. Given $\lambda \in \R$, define the following sets,
\begin{align*}
L_{\lambda} = \big\{t\e_1+\lambda \e_2~ | ~ t \in \R\big\}, \quad S= \big\{ x\in \R^d~|~\|x\|=1\big\},
\end{align*}
that is, a line and the unit Euclidean sphere in $\R^d$. In this case, the Douglas-Rachford operator is given explicitly by the following formula,
\begin{align}\label{for dr}
\T x = \frac{\langle x,\e_1\rangle}{\|x\|}\e_1 + \left(\left(1-\frac 1{\|x\|}\right)\langle x,\e_2\rangle+ \lambda\right) \e_2 + \left(1-\frac 1 {\|x\|}\right)\sum_{j=3}^d\langle x,\e_j\rangle\e_j,
\end{align}
where here and in what follows, $\langle \cdot\, , \cdot\rangle$ denotes the standard inner product in $\R^d$. Define also the following sets
\begin{align}
\nonumber H_0 & = \big\{x\in \R^d~\big|~ \langle x,\e_1\rangle = 0\big\}\setminus\{0\}, 
\\
\label{def H} H_+ & = \big\{x\in \R^d~\big|~ \langle x,\e_1\rangle > 0\big\},
\\
\nonumber H_- & = \big\{x\in \R^d~\big|~ \langle x,\e_1\rangle < 0\big\}.
\end{align}
By~\eqref{for dr}, all three sets are invariant under $\T$. Assuming that $\lambda \in [0,1]$ (the case $\lambda\in [-1,0]$ is completely analogous), the only fixed points of $\T$ are the intersection points of $L_{\lambda}$ and $S$,
\begin{align}\label{def x*}
x^* = \sqrt{1-\lambda^2}\,\e_1+\lambda \e_2 \quad\text{and} \quad x_* = -\sqrt{1-\lambda^2}\,\e_1+\lambda \e_2.
\end{align}
In~\cite{BS11} it was shown that when $\lambda\in (0,1)$, the Douglas-Rachford iteration is locally norm convergent around the intersection points, and later in~\cite{AB13} an explicit domain of convergence was given. In the case $\lambda=0$, it was shown in~\cite{BS11} that in fact we have global convergence on $H_-\cup H_+$. In the case $\lambda=1$, it was also shown in~\cite{BS11} that if $x_0\in H_-\cup H_+$, then the Douglas-Rachford iteration converges to a point of the form $\hat{y} \ \e_2$, where $\hat y \in (1,\infty)$. Finally, it was shown in~\cite{BS11}, that when $\lambda >1$ and $x_0\in H_-\cup H_+$ or when $x_0 \in H_0$, the Douglas-Rachford iteration is divergent. In~\cite{Ben15} it was shown that for every $\lambda\in (0,1)$, the Douglas-Rachford iteration converges globally in norm whenever $x_0\in H_-\cup H_+$. From~\eqref{for dr}, it follows that the behaviour of $\T$ on $H_+$ and $H_-$ are similar. Thus, it is enough to consider the case where the initial point is in $H_+$. In this case, the only intersection point we need to consider is $x^*$. Define the following domain $\del \subseteq \R^d$, 
\begin{align}\label{def del}
\del = \left\{x\in \R^d~\big|~ \langle x,\e_1\rangle \in (0,1]\right\}.
\end{align}
It follows from~\eqref{for dr} that 
\begin{align}\label{prop invariance}
\T(\del) \subseteq \T(H_+)\subseteq \del .
\end{align} 


In~\cite{Ben15}, an ingenious construction of a Lyapunov function for $\T$ was presented, which implied global convergence of the iteration~\eqref{iterate seq}. Given $x\in \R^d$ and $t>0$, let $B(x,t)$ denote the open Euclidean ball centred at $x$ with radius $t$. The following is the main result of~\cite{Ben15}.


\begin{thm}[\cite{Ben15}]\label{thm Ben}
Assume that $\lambda \in [0,1)$. Define the function $F:\del \to \R$ as follows,
\begin{align}\label{def F}
F(x) = \frac 1 2 \|x - \lambda \e_2\|^2 -\lambda \log\left(1+\sqrt{1-\langle x,\e_1\rangle^2}\right) + \lambda\sqrt{1-\langle x,\e_1\rangle^2}+ (\lambda-1)\log\langle x,\e_1\rangle.
\end{align} 
Then $F$ satisfies $F(x^*) < F(x)$ for all $x\in \del\setminus\{x^*\}$ and $F(\T x) \le F(x)$ for all $x\in \del$. Moreover, for every $t >0$, 
\begin{align*}
\sup_{x\in \del\setminus B(x^*,t)}\big[F(\T x) - F(x)\big] < 0.
\end{align*}
In particular, $F(\T x) = F(x)$ if and only if $x=x^*$.
\end{thm}


Note that in~\cite{Ben15}, the main result is stated for $\lambda \in (0,1)$, but it is in fact true for the case $\lambda = 0 $ as well.

If $K\subseteq H_+$ and for $n=1,2,\dots$ we define $f_n:K \to \R$ by $f_n(x) = F(\T^nx)-F(x^*)$, then by Theorem~\ref{thm Ben} $\{f_n\}_{n=1}^{\infty}$ is a decreasing sequence of continuous functions which converges point-wise to 0. Therefore, if $K$ is compact, then by Dini's convergence theorem, $\{f_n\}_{n=1}^{\infty}$ converges uniformly to 0 on $K$. This in particular implies that $\T^n x \stackrel{n\to \infty}{\longrightarrow} x^*$ uniformly for $x\in K$. Here and in what follows any convergence of vectors means convergence in the Euclidean norm on $\R^d$. However, using Theorem~\ref{thm converge} below, it is shown that on compact sets in $H_+$ we obtain a type of convergence which is stronger than uniform convergence. See Section~\ref{sec state} below for the exact formulation.


The rest of the note is organised as follows. In Section~\ref{sec stable}, we recall some preliminaries and notations from the theory of discrete time dynamical systems. Then, in Section~\ref{sec state}, we state the main results in the note, which are then proved in Section~\ref{sec proofs}.


\section{Preliminaries and statement of the main results}

\subsection{Stability of discrete time dynamical systems}\label{sec stable}

Assume that $\d$ is a set in $\R^d$, and let $T:\d \rightrightarrows \d$ be a set-valued map from $\d$ to subsets of $\d$. For $n \in \Z_+ = \{ 0,1,2,\dots\}$, consider the difference inclusion with initial condition
\begin{align}\label{diff eq}
x_{n+1} \in T x_n, ~~ x_0 \in \d.
\end{align}
Let $\S(x, T)$ be the set of solutions to~\eqref{diff eq} with $x_0=x$, and let $\phi(x,n)\in \S(x, T)$ denote a solution to~\eqref{diff eq}, that is, $\phi:\d\times \Z_+ \to \d$ is such that $\phi(x,0) = x$ and $\phi(x,n+1) \in T(\phi(x,n))$ for all $n\in \Z_+$. For $K\subseteq \d$ let
\begin{align*}
\S(K, T) = \bigcup_{x\in K}\S(x, T).
\end{align*}


Next, recall some definitions regarding the stability of the system~\eqref{diff eq}. Let $\R_+ = [0,\infty)$. A function $\b: \R_+\times \R_+ \to \R_+$ is said to belong to class $\kl$ if for every $t\ge 0$, $\b(\cdot,t)$ is continuous, strictly increasing, and $\b(0,t)=0$, and also for every $s \ge 0$, $\b(s,\cdot)$ is non-increasing, and satisfies $\beta(s,t) \stackrel{t\to \infty}{\longrightarrow} 0$. We have the following definition.

\begin{defin}
Assume that $\d\subseteq \R^d$ and $\omega_1,\omega_2:\d \to \R_+$ are continuous functions. The difference inclusion~\eqref{diff eq} is said to be $\kl$-stable with respect to $(\omega_1,\omega_2)$ if there exists $\beta \in \kl$ such that for every $x\in \d$, every $\phi\in \S(x)$, and every $n\in \Z_+$, 
\begin{align*}
\omega_1(\phi(x,n)) \le \beta(\omega_2(x),n).
\end{align*}
\end{defin}


Let $\sigma:\d\to \R_+$ be such that $B[x,\sigma(x)] \subseteq \d$ for all $x\in \d$, where here and in what follows $B[x,r]$ denotes the closed ball around $x$ with radius $r$ with respect to the Euclidean norm. Given a set $K\subseteq \d$, define
\begin{align*}
K_\sigma = \bigcup_{x\in K}B[x,\sigma(x)].
\end{align*}
Note that $K_\sigma \subseteq \d$. Given an operator $T:\d \rightrightarrows \d$, define also $T_{\sigma}$, the $\sigma$-perturbation of $T$,
\begin{align*}
T_{\sigma}x = \bigcup_{y\in T(B[x,\sigma(x)])}B[y, \sigma(y)]. 
\end{align*}
Note that if $K\subseteq \d$, then $T_\sigma(K) = (T(K_\sigma))_\sigma$. Denote by $\S_{\sigma}(x, T)$ the set of solutions to the perturbed difference inclusion $x_{n+1}\in T_{\sigma}x_n$ with initial condition $x_0 = x$. Note that in particular, we have $\S_0(x, T) = \S(x, T)$, where here 0 denotes the constant zero function. As before, for $K\subseteq \d$, let
\begin{align*}
\S_{\sigma}(K, T) =  \bigcup_{x\in K}\S_{\sigma}(x, T).
\end{align*}


Given a continuous function $\omega_1:\d \to \R_+$ define the following set
\begin{align*}
\mathcal A_{\sigma} = \Big\{x\in \d~\Big|~ \sup_{n\in \Z_+}\sup_{\phi\in \S_{\sigma}(x)}\omega_1(\phi(x,n)) = 0\Big\},
\end{align*}
and let 
\begin{align}\label{def a}
\mathcal A = \Big\{x\in \d~\Big|~ \sup_{n\in \Z_+}\sup_{\phi\in \S(x)}\omega_1(\phi(x,n)) = 0\Big\}.
\end{align}

Next, recall the notion of robust $\kl$-stability.


\begin{defin}\label{def robust kl}
Assume that $\d \subseteq \R^d$ and $\omega_1,\omega_2 :\d \to \R_+$ are continuous functions. The difference inclusion~\eqref{diff eq} is said to be robustly $\kl$-stable with respect to $(\omega_1,\omega_2)$ on $\d$ if there exists a continuous function $\sigma:\d\to\R_+$ such that
\begin{enumerate}
\item \label{ball contained} For all $x\in \d$, $B[x,\sigma(x)] \subseteq \d$;
\item \label{always positive} For all $x\in \d \setminus \mathcal A$, $\sigma(x) >0$;
\item \label{stable set} $\mathcal A_{\sigma} = \mathcal A$;
\item The difference inclusion $x_{n+1} \in T_{\sigma}x_n$ is $\kl$-stable with respect to $(\omega_1,\omega_2)$ on $\d$.
\end{enumerate}
\end{defin}


\begin{remark}\label{rmk smaller}
Note that if $x_{n+1} \in T_{\sigma}x_n$ is $\kl$-stable with respect to $(\omega_1,\omega_2)$ and $\tau:\d\to \R_+$ is such that $\tau \le \sigma$ on $\d$, then $x_{n+1} \in T_{\tau}x_n$ is $\kl$-stable with respect to $(\omega_1,\omega_2)$ on $\d$ as well, since $\S_\tau(K,T) \subseteq \S_\sigma(K,T)$. \qede
\end{remark}


A function $\alpha:\R_+\to \R_+$ is said to belong to class $\mathcal K$ if it is continuous, strictly increasing, and $\alpha(0) = 0$. A function $\alpha :\R_+\to \R_+$ is said to belong to class $\mathcal K_{\infty}$ if $\alpha \in \mathcal K$ and in addition $\lim_{t\to \infty}\alpha(t) = \infty$. A function $\alpha :\R_+\to \R_+$ is said to be positive definite if $\alpha(t) = 0$ if and only if $t=0$. 


Recall next the following notion of a Lyapunov function.


\begin{defin}\label{def lyap}
Assume that $\d \subseteq \R^d$ and $\omega_1,\omega_2:\d \to \R_+$ are continuous functions. A function $V:\d \to \R_+$ is said to be a Lyapunov function with respect to $(\omega_1,\omega_2)$ on $\d$ for the difference inclusion~\eqref{diff eq} if there exist $\alpha_1,\alpha_2\in \mathcal K_{\infty}$ and a continuous positive definite function $\alpha$ such that for all $x\in \d$,
\begin{eqnarray}
\label{condition between} &\alpha_1(\omega_1(x)) \le V(x) \le \alpha_2(\omega_2(x)),&
\\
\label{condition alpha} &\sup_{y\in Tx}V(y) \le V(x) -\alpha(V(x)),&
\\
\label{condition iff} &V(x) = 0 \iff x\in \mathcal A.&
\end{eqnarray}
\end{defin}


It is known that there is a close connection between the stability properties of a dynamical system and the existence and properties of a Lyapunov function. Originally this was known for continuous time dynamical systems, but results for the discrete time case have also been obtained. See for example the survey~\cite{Kel15}. In particular, the following result is Theorem 2.8 in~\cite{KT05}.

\begin{thm}[\cite{KT05}]\label{thm kl}
Assume that $\d \subseteq \R^d$ is open, and $T:\d \rightrightarrows \d$ is such that $Tx$ is compact and non-empty for all $x\in \d$. Assume also that there exists a continuous Lyapunov function on $\d$ with respect to two continuous functions $\omega_1,\omega_2:\d \to \R_+$. Then the difference inclusion~\eqref{diff eq} is robustly $\kl$-stable with respect to $(\omega_1,\omega_2)$.
\end{thm}


\subsection{Statement of the main results}\label{sec state}

Using the explicit construction from Theorem~\ref{thm Ben} in Theorem~\ref{thm kl}, we obtain the following result.


\begin{thm}\label{thm kl dr}
Assume that $\lambda \in [0,1)$ and let $H_+$ be the domain defined in~\eqref{def H}. Then there exists a continuous Lyapunov function $V:H_+ \to \R_+$ in the sense of Definition~\ref{def lyap} such that the Douglas-Rachford iteration $x_{n+1} = \T x_n$, $x_0 \in H_+$, is robustly $\kl$-stable with respect to $\omega_1 = \omega_2 = V$.
\end{thm}


Theorem~\ref{thm kl dr} implies the following convergence result for the Douglas-Rachford iteration.


\begin{thm}\label{thm converge}
Assume that $\lambda \in [0,1)$ and let $H_+$ be the domain defined in~\eqref{def H}. Then there exists a function $\tau: H_+\to \R_+$ which satisfies $\tau(x) >0$ for all $x\in H_+\setminus\{x^*\}$, and such that for every compact set $K\subseteq H_+$, $\phi \in \S_\tau(K,\T)$ converges uniformly to $x^*$, that is,
\begin{align*}
\lim_{n\to \infty}\sup_{\phi \in \S_{\tau}(K,\T)}\|\phi(x,n) - x^*\| = 0.
\end{align*}
\end{thm}


Theorem~\ref{thm converge} says that if we consider paths which are `close enough' to the paths resulting from the Douglas-Rachford iteration, we still have uniform convergence.


\begin{remark}
Regarding the `boundary' case $\lambda=1$, it was shown in~\cite{BS11} that in this case we have global convergence on $H_+$, even though it need not converge to the intersection point. See Theorem 6.12 in~\cite{BS11}. It would be interesting to see whether a result similar to Theorem~\ref{thm kl dr} or Theorem~\ref{thm converge} can be obtained in this case. \qede
\end{remark}


\section{Proof of Theorem~\ref{thm kl dr} and Theorem~\ref{thm converge}}\label{sec proofs}

For the domain $H_+$ as defined in~\eqref{def H}, define $V: H_+ \to \R_+$,
\begin{align}\label{def V}
V(x) = F(\T x) - F(x^*).
\end{align}
By~\eqref{prop invariance}, $\T x\in \del$ whenever $x\in H_+$, where $\del$ is defined as in~\eqref{def del}. Thus, $V$ is well defined, and by Theorem~\ref{thm Ben}, $V(x) \ge 0$. We would like to show that $V$ is a continuous Lyapunov function that satisfies the conditions of Definition~\ref{def lyap}. The following proposition shows that condition~\eqref{condition alpha} holds for this choice of $V$.


\begin{prop}\label{prop alpha}
Let $\lambda \in [0,1)$, and let $V:H_+ \to \R_+$ be defined as in~\eqref{def V}. Then there exists a continuous, positive definite function $\alpha:\R_+\to \R_+$ such that for every $x\in H_+$,
\begin{align*}
V(\T x) \le V(x) -\alpha(V(x)).
\end{align*}
\end{prop}

\begin{proof}
Let $\del$ be the domain defined in~\eqref{def del}, and define $U, W:\del \to \R_+$,
\begin{align}\label{def U}
U(x) = F(x) - F(x^*),
\end{align}
and
\begin{align}\label{def w}
W(x) = U(x) - U(\T x) = F(x) - F(\T x).
\end{align}
By Theorem~\ref{thm Ben}, both $U$ and $W$ are continuous and positive on $\del$, and are equal to 0 if and only if $x=x^*$. Define also $g:\R_+\to \R_+$,
\begin{align*}
g(t) = \inf_{x\in \del \setminus B(x^*,t)}W(x).
\end{align*}
Clearly $g$ is non-decreasing. Also, by Theorem~\ref{thm Ben}, $g$ is non-negative and $g(t) = 0$ if and only if $t=0$, and so $g$ is positive definite. Note that by~\eqref{def w}, combined with~\eqref{for dr} and~\eqref{def F},
\begin{align}\label{with log} 
W(x) = \psi(x) + (\lambda-1)\left(\log \langle x, \e_1\rangle - \log\left(\frac{\langle x, \e_1\rangle}{\|x\|}\right)\right) = \psi(x)+ (\lambda -1) \log \|x\| ,
\end{align}
where $\psi$ is a continuous function on $\big\{x\in \R^d~\big|~ -1 \le \langle x,\e_1 \rangle \le 1\big\}$. Also, $\log \|x\|$ is continuous on $\R^d\setminus \{0\}$ which contains $\del$. Since $g$ is non-decreasing, in order to show that it is continuous it is enough to show that for every $\eps >0$ and $t\ge 0$, there exists $\delta >0$ such that $g(t+\delta) \le g(t) + \eps$. Indeed, let $t \ge 0$ and $\eps>0$. Assume that $x\in \del\setminus B(x^*,t)$ is such that 
\begin{align}\label{small with eps}
W(x) \le g(t) + \eps.
\end{align} 
Since by~\eqref{with log} $W$ is continuous on $\del$ and $W(x) \stackrel{x\to 0}{\longrightarrow} \infty$, it follows that we may choose $x\in \mathrm{int}\big(\del\setminus B(x^*,t)\big)$, the interior of $\del \setminus B(x^*,t)$. Given $\delta >0$, define $x_\delta$ as follows,
\begin{align*}
x_\delta = x^* + \left(1+\frac{\delta}{\|x-x^*\|}\right)(x-x^*).
\end{align*}
Then  $\|x_\delta - x\| = \delta$ and $\|x_\delta - x^*\| = \|x-x^*\| + \delta$. Since $x\in \del\setminus B(x^*,t)$, we have $\|x-x^*\| \ge t$ and so $\|x_\delta - x^*\| \ge t+\delta$. Since $x\in \mathrm{int}\big(\del\setminus B(x^*,t)\big)$, if $\delta$ is sufficiently small then we may assume that $x_\delta \in \del$. Altogether, we have that $x_\delta \in \del \setminus B(x^*,t+\delta)$. Since $W$ is continuous on $\del$, if $\delta$ is sufficiently small, we have $W(x_\delta) \le W(x) + \eps$. Therefore,
\begin{align*}
g(t+\delta) = \inf_{x\in \del\setminus B(x^*,t+\delta)}W(x) \le W(x_\delta) \le W(x) + \eps \stackrel{\eqref{small with eps}}{\le} g(t) + 2\eps.
\end{align*}
Since $\eps>0$ is arbitrary, this shows that $g$ is continuous.


Next, define $\alpha:\R_+\to \R_+$,
\begin{align*}
\alpha(t) = \inf \big\{g(\|y-x^*\|)~\big|~ y\in \del, ~~U(y) \ge t\big\}.
\end{align*}
Then $\alpha(0) = 0$, and $\alpha$ is non-negative and non-decreasing. We would also like to show that $\alpha$ is positive definite and continuous. Indeed, assume that $\alpha(0) = 0$. Then for every $\eps>0$, there exists $y_\eps \in \del$ such that $g(\|y_\eps - x^*\|) \le \eps$. Since $g$ is positive definite and continuous, it follows that $y_\eps \stackrel{\eps \to 0}{\longrightarrow} x^*$. Thus, we must have $t=0$. This shows that $\alpha$ is positive definite. In order to prove that $\alpha$ is continuous, let $\eps >0$ and $t\ge 0$. Let $y\in \del$ with $U(y) \ge t$ be such that $g(\|y-x^*\|) \le \alpha(t) + \eps$. Let $\delta>0$, and choose $y_\delta \in \del$ such that $U(y_\delta) \ge t+\delta$. Since $U$ is continuous on $\del$, it follows that we can choose $y_\delta$ such that $y_\delta \stackrel{\delta\to 0}{\longrightarrow} y$. Since $g$ is continuous, if $\delta$ is sufficiently small, then
\begin{align*}
g(\|y_\delta - x^*\|) \le g(\|y-x^*\|) + \eps \le \alpha(t)+2\eps,
\end{align*}
and taking the infimum over the left hand side gives
\begin{align*}
\alpha(t+\delta) \le \alpha(t) + 2\eps.
\end{align*}
Since $\eps>0$ is arbitrary and since $\alpha$ is non-decreasing, it follows that $\alpha$ is continuous. 


To conclude the proof, note that
\begin{align*}
U(x) - U(\T x) \ge \inf_{y\in \del\setminus B(x^*, \|x-x^*\|)}\big[U(y) - U(\T y)\big] = g(\|x-x^*\|) \ge \alpha(U(x)).
\end{align*}
By~\eqref{prop invariance}, $\T x\in \del$ whenever $x\in H_+$. Thus,
\begin{align*}
V(x) \stackrel{\eqref{def V}}{=} F(\T x) - F(x^*) \stackrel{\eqref{def U}}{=} U(\T x).
\end{align*}
Therefore,
\begin{align*}
V(x) - V(\T x) = U(\T x) - U(\T^2 x) \ge \alpha(U(\T x)) = \alpha(V(x)),
\end{align*}
and the proof is complete.
\end{proof}


We are now in a position to prove Theorem~\ref{thm kl dr} and Theorem~\ref{thm converge}.


\begin{proof}[Proof of Theorem~\ref{thm kl dr}]
By the definition of $V$~\eqref{def V}, together with~\eqref{prop invariance} and~\eqref{def F}, it follows that $V$ is continuous on $H_+$, which is an open set. Therefore, since we choose $\omega_1=\omega_2=V$, all functions are continuous. Clearly by the choice of $\omega_1$, $\omega_2$, it follows that~\eqref{condition between} holds with $\alpha_1(t) = \alpha_2(t) = t$, which is of class $\mathcal K_\infty$. Condition~\eqref{condition alpha} holds by Proposition~\ref{prop alpha}. Finally, it follows directly from~\eqref{for dr} that if $x\in H_+$ is such that $\T x = x^*$, then $x=x^*$. Thus, by Theorem~\ref{thm Ben} and~\eqref{def V}, we have that $V(x)=0 \iff x=x^*$. On the other hand, by the definition of $\mathcal A$~\eqref{def a} and the choice of $\omega_1$, 
\begin{align*}
x\in \mathcal A \iff \sup_{n\in \Z_+}\sup_{\phi \in \S(x)}V(\phi(n,x)) = 0 \iff \phi(n,x) = x^*, \forall n \in \Z_+, \forall \phi \in \S(x).
\end{align*}
Since the only fixed point of $\T$ in $H_+$ is $x^*$, it follows that $x\in \mathcal A \iff x=x^*$. Hence,~\eqref{condition iff} holds as well, and so $V$ is a continuous Lyapunov function on $H_+$ in the sense of Definition~\ref{def lyap} with respect to $(\omega_1,\omega_2)$. Therefore, by Theorem~\ref{thm kl}, the Douglas-Rachford iteration~\eqref{iterate seq} is robustly $\kl$-stable on $H_+$, and this completes the proof.
\end{proof}


\begin{proof}[Proof of Theorem~\ref{thm converge}]
By Theorem~\ref{thm kl dr}, the Douglas-Rachford iteration~\eqref{iterate seq} is robustly $\kl$-stable on $H_+$. Therefore, there exists $\sigma:H_+\to \R_+$ which satisfies $\sigma(x) = 0 \iff x=x^*$ and such that the $\sigma$-perturbation of $\T$ is $\kl$-stable on $H_+$. Define $\tau:H_+\to \R_+$,
\begin{align}\label{def tau}
\tau(x) = \min\left\{\sigma(x), \frac 1 2\langle x,\e_1\rangle\right\}.
\end{align}
Note that $\tau \le \sigma$ and also $\tau(x) = 0 \iff x=x^*$. Note also that by Remark~\ref{rmk smaller}, the $\tau$-perturbation of $\T$ is $\kl$-stable. Thus, by Theorem~\ref{thm kl dr}, there exists a function $\beta$ of class $\kl$, such that for all $x\in K$, all $\phi \in \S_{\tau}(x, \T)$, and all $n \in \Z_+$,
\begin{align}\label{bound with beta}
V(\phi(x,n)) \le \b(V(x), n).
\end{align}


If $K\subseteq H_+$ is compact, it is bounded and there exists $b\in (0,\infty)$ such that $\inf_{x\in K}\langle x,\e_1\rangle \ge b$. Therefore, by the definition of $\tau$, it follows that $K_\tau$ is bounded and $\inf_{x\in K_\tau}\langle x,\e_1\rangle \ge \frac 1 2b >0$. This means that $\overline{K}_\tau$, the closure of $K_\tau$, is compact and satisfies $\overline{K}_\tau \subseteq H_+$. Since $V$ is continuous on $H_+$, there exists $M\in (0,\infty)$ such that $\sup_{x\in K_\tau}V(x) = M$. Therefore, given $n\in \Z_+$,
\begin{align}\label{bound separate}
\sup_{\phi\in \S_{\tau}(K,\T)}V(\phi(x,n)) \stackrel{\eqref{bound with beta}}{\le} \sup_{x\in K_\tau}\b(V(x),n) \stackrel{(*)}{\le} \b\Big(\,\sup_{x\in K_\tau}V(x), n\Big) = \b\left(M,n\right),
\end{align}
where in ($*$) we used the fact that $\beta(\cdot, n)$ is increasing for all $n\in \Z_+$. Thus, since $\beta(M, n)\stackrel{n\to \infty}{\longrightarrow}0$ for all $M\in \R_+$,~\eqref{bound separate} gives
\begin{align}\label{v to zero}
\lim_{n\to \infty}\sup_{\phi\in \S_{\tau}(K,\T)}V(\phi(x,n)) = 0.
\end{align}
Assume that there exist $\{x_n\}_{n=0}^{\infty}\subseteq K$, $\{\phi_n\}_{n=0}^{\infty} \subseteq S_\tau(K,\T)$, and $\eps >0$ such that 
\begin{align*}
\inf_{n \in \Z_+}\|\phi_n(x_n,n)-x^*\| \ge \eps. 
\end{align*}
Then since $\T$ is continuous on $H_+$ and since $\T x = x \iff x=x^*$, there exists $\eps'>0$ such that
\begin{align*}
\inf_{n \in \Z_+}\|\T(\phi_n(x_n,n))-x^*\| \ge \eps',
\end{align*}
and so by Theorem~\ref{thm Ben}, there exists $\eps''>0$ such that 
\begin{align*}
\inf_{n \in \Z_+}V(\phi_n(x_n,n)) = \inf_{n \in \Z_+}\big[F(\T(\phi_n(x_n,n)))-F(x^*)\big] \ge \eps''.
\end{align*}
But this is a contradiction to~\eqref{v to zero}. Therefore, 
\begin{align*}
\lim_{n\to \infty} \sup_{\phi \in S_\tau(K,\T)}\|\phi(x,n)-x^*\| = 0,
\end{align*}
and this completes the proof.
\end{proof}

\begin{remark}
The choice of $1/2$ in the definition of $\tau$~\eqref{def tau} is not of any significance. One can choose any number in $(0,1)$. \qede
\end{remark}


\subsection*{Acknowledgments} Many thanks to Bj\"orn R\"uffer for some helpful discussions, and in particular for bringing to our attention the results of~\cites{Kel15, KT05}. Many thanks also to Brailey Sims for some very valuable comments on this note.

\end{document}